\documentclass[a4paper,11pt]{amsart}
\usepackage{mathptmx}
\usepackage{amsmath}
\usepackage{amscd}
\usepackage{amssymb, bm}
\usepackage{amsthm}
\usepackage{xspace}
\usepackage[all,tips]{xy}
\usepackage[dvips]{graphicx}
\usepackage{verbatim}
\usepackage{syntonly}
\usepackage{hyperref}
\usepackage{graphics, fullpage,color, epsfig,url}
\usepackage{indentfirst}
\usepackage{esint}
\usepackage{enumitem}
\usepackage[dvipsnames]{xcolor}
\usepackage{soul, amsaddr}
\usepackage{tensor}
\usepackage{amsrefs, eucal}
\usepackage{bbm}
\usepackage{ulem}
\usepackage{tikz, wrapfig, epsfig}

\providecommand{\MR}{\relax\ifhmode\unskip\space\fi MR }

\providecommand{\href}[2]{#2}

\theoremstyle{plain}
\newtheorem{thm}{Theorem}%[section]

\newtheorem{lem}{Lemma}
\newtheorem{prop}{Proposition}

\newtheorem*{clm}{Claim}

\theoremstyle{remark}
\newtheorem{rem}[thm]{Remark}

\newcommand{\disp}{\displaystyle}

\DeclareMathOperator{\di}{div}

\DeclareMathOperator{\sgn}{sgn}
\DeclareMathOperator{\loc}{loc}

\newcommand{\eps}{\varepsilon}
\newcommand{\vp}{\varphi}

\newcommand{\al}{\alpha}
\newcommand{\be}{\beta}
\newcommand{\ga}{\gamma}
\newcommand{\de}{\delta}

\newcommand{\Ga}{\Gamma}
\newcommand{\te}{\theta}
\newcommand{\la}{\lambda}

\newcommand{\Om}{\Omega}

\newcommand{\nid}{\noindent}

\newcommand{\iny}{\infty}
\newcommand{\del}{ \partial}
\newcommand{\su}{\subset}
\newcommand{\LP}{\Delta}
\newcommand{\gr}{\nabla}

\newcommand{\norm}[1]{\left\| #1\right\|}
\newcommand{\innp}[1]{\left< #1 \right>}
\newcommand{\abs}[1]{\left\vert#1\right\vert}

\newcommand{\set}[1]{\left\{#1\right\}}
\newcommand{\brac}[1]{\left[#1\right]}
\newcommand{\pr}[1]{\left( #1 \right) }

\newcommand{\WT}[1]{\ensuremath{\widetilde{#1}}}

\newcommand{\N}{\ensuremath{\mathbb{N}}}

\newcommand{\R}{\ensuremath{\mathbb{R}}}

\newcommand{\C}{\ensuremath{\mathbb{C}}}

\def\XXint#1#2#3{{\setbox0=\hbox{$#1{#2#3}{\int}$}
\vcenter{\hbox{$#2#3$}}\kern-.5\wd0}}

\numberwithin{equation}{section}
\date{}

\begin{document}

\title{On Landis' conjecture in the plane for \\ real-valued potentials with decay}
\author[Davey]{Blair Davey}
\address{Department of Mathematical Sciences, Montana State University, Bozeman, MT, 59717}
\email{blairdavey@montana.edu}
\thanks{B. D. is supported in part by the NSF CAREER DMS-2236491 and NSF LEAPS-MPS grant DMS-2137743.}

\begin{abstract}
We investigate the quantitative unique continuation properties of real-valued solutions to planar Schr\"odinger equations with potential functions that exhibit pointwise decay at infinity.
That is, for equations of the form $-\LP u + V u = 0$ in $\R^2$, where $\abs{V(z)} \lesssim \innp{z}^{-N}$ for some $N > 0$, we prove that real-valued solutions satisfy exponential decay estimates with a rate that depends explicitly on $N$.
Examples show that the estimates established here are essentially sharp.
The case of $N = 0$ corresponds to the Landis conjecture, which was proved for real-valued solutions in the plane in \cite{LMNN20}, while the case of $N < 0$ was previously investigated by the author in \cite{Dav24}.
Here, the proof techniques rely on the ideas presented in \cite{LMNN20} combined with conformal transformations and an iteration scheme. 
\end{abstract}

\maketitle

\section{Introduction}

This article is concerned with the quantitative unique continuation of solutions to elliptic Schr\"odinger equations in the plane.
We consider equations with potentials that exhibits pointwise decay at infinity.
The main result is a nearly sharp estimate for the optimal rate of decay at infinity for real-valued solutions.
This theorem may be interpreted as a quantitative Landis-type theorem.

In the late 1960s, E.~M.~Landis \cite{KL88} conjectured that if $u$ and $V$ are bounded functions that satisfy
\begin{equation}
\label{ePDE}
-\LP u + V u = 0 \; \text{ in } \, \R^n,
\end{equation}
and $u$ decays faster than exponentially, i.e., $\abs{u(x)} \lesssim \exp\pr{- c \abs{x}^{1+}}$, then it follows that $u \equiv 0$.
This conjecture was later disproved by Meshkov \cite{M92} who constructed non-trivial complex-valued functions $u$ and $V$ that solve \eqref{ePDE} in $\R^2$, where $V$ is bounded and $\abs{u(x)} \lesssim \exp\pr{- c \abs{x}^{4/3}}$. 
Using Carleman estimate techniques, Meshkov also proved a \textit{qualitative unique continuation} result: 
If \eqref{ePDE} holds, where $V$ is bounded and $u$ satisfies a decay estimate of the form $\abs{u\pr{x}} \lesssim \exp\pr{- c \abs{x}^{4/3+}}$, then necessarily $u \equiv 0$.

In their work on Anderson localization \cite{BK05}, Bourgain and Kenig established a quantitative version of Meshkov's result. 
As a first step in their proof, they used three-ball inequalities derived from Carleman estimates to establish \textit{order of vanishing} estimates for local solutions to Schr\"odinger equations of the form \eqref{ePDE}.
Then, through a scaling argument, they proved a \textit{quantitative unique continuation} result.
More specifically, they showed that if $u$ and $V$ are bounded and satisfy \eqref{ePDE}, and $u$ is normalized so that $\abs{u(0)} \ge 1$, then for sufficiently large values of $R$,
\begin{equation}
 \inf_{|x_0| = R}\norm{u}_{L^\iny\pr{B_1(x_0)}} \ge \exp{(-CR^{\be}\log R)},
\label{est}
\end{equation} 
where $\be = \frac 4 3$.
Since $ \frac 4 3 > 1$, the constructions of Meshkov, in combination with the qualitative and quantitative unique continuation theorems just described, indicate that Landis' conjecture cannot be true for complex-valued solutions in $\R^2$.
However, at the time, Landis' conjecture still remained open in the real-valued and higher-dimensional settings.

The first significant step towards resolving Landis' conjecture in the real-valued planar setting was made by Kenig, Silvestre and Wang in \cite{KSW15} where they proved a quantitative form of Landis' conjecture under the assumption that the zeroth-order term satisfies $V \ge 0$ a.e.
The techniques in \cite{KSW15} exploit the relationship between real-valued solutions to second-order elliptic PDEs in the plane and solutions to complex-valued Beltrami equations.
Using similar ideas, analogous results were established in the settings with drift terms \cite{KW15}, variable coefficients \cite{DKW17}, singular lower-order terms \cite{DW20}, and when $V_-$ exhibits decay at infinity \cite{DKW19, Dav20a}.
Then, in \cite{LMNN20}, Logunov, Malinnikova, Nadirashvili, and Nazarov proved Landis' conjecture in the real-valued planar setting.
Their proof uses the nodal structure of the domain along with a domain reduction technique to eliminate any sign condition on the zeroth-order term.
Many of the ideas from \cite{LMNN20} are used in this article.

In \cite{Dav14}, we studied the quantitative unique continuation properties of solutions to elliptic equations of the form 
$$\LP u + W \cdot \gr u + V u = \la u \; \text{ in } \; \R^n,$$
where $V$ and $W$ exhibit pointwise decay at infinity, and $\la \in \C$.
With $\innp{x} = \sqrt{1 + \abs{x}^2}$, it was shown that if $\abs{V\pr{x}} \lesssim \innp{x}^{-N}$ and $\abs{W\pr{x}} \lesssim \innp{x}^{-P}$ for $N, P \ge 0$, then the quantitative estimate \eqref{est} holds with $\be = \max \set{1, \frac{4-2N}{3}, 2 - 2P}$ and $\log R$ replaced by a different slowly-decaying function.
These quantitative estimates were generalized in \cite{LW14}, where Lin and Wang proved analogous estimates for solutions to the corresponding equations with variable-coefficient leading terms.
Qualitative estimates for similar equations are given in \cite{CS97}.
The constructions presented in \cite{Dav14, Dav15} show that the estimates described in this paragraph are sharp.

In \cite{Dav24}, we studied quantitative unique continuation at infinity for real-valued solutions to \eqref{ePDE} when $n = 2$ and the potential exhibits growth at infinity, i.e., $\abs{V(z)} \lesssim \abs{z}^N$ for $N > 0$.
The techniques in that article rely heavily on the ideas in \cite{LMNN20} and careful scaling arguments.
Here, we address the more difficult setting where the potential exhibits decay at infinity.
The precise statement of the main theorem is as follows.

\begin{thm}[Main Theorem]
\label{mainThm}
For some $a_0 \ge 1$ and $N \in (0, 2)$, let $V : \R^2 \to \R$ satisfy 
\begin{equation}
\label{Vbound}
\abs{V(w)} \le a_0^2 \innp{w}^{-N}.
\end{equation}
Assume that $u : \R^2 \to \R$ is a solution to 
\begin{equation}
\label{ellipEq}
- \LP u + V u = 0 \, \text{ in } \R^2
\end{equation}
with the properties that 
\begin{equation}
\label{solNorm}
\abs{u(0)} = 1
\end{equation}
and for each $w \in \R^2$,
\begin{equation}
\label{uBound}
\abs{u(w)} \le \exp\pr{c_0 \abs{w}^{1 - \frac N 2}}.
\end{equation}
For every $\eps \in \pr{0, \frac N 2}$, there exists ${R}_0(N, a_0, c_0, \eps) > 0$ so that whenever $R \ge {R}_0$, it holds that
\begin{equation}
\label{uLower}
\inf_{\abs{w_0} = R} \norm{u}_{L^\iny\pr{B_1(w_0)}} \ge \exp\pr{-R^{1 - \frac N 2 + \eps}}.
\end{equation}
\end{thm}

The results of \cite{Dav14} prove an estimate of the form \eqref{est} with $\be = \frac{4 - 2N}{3} = \frac 4 3 \pr{1 - \frac N 2} > 1 - \frac N 2 + \eps$.
In that article, the assumptions are the same as those in Theorem \ref{mainThm}, except that $u$ may be complex-valued.
Thus, as in the case of bounded $V$, Theorem \ref{mainThm} illustrates that better estimates hold in the real-valued planar setting.

As illustrated by the following example, Theorem \ref{mainThm} is sharp (up to $\eps$) for all $N \in (0, 2)$.
Fix $N \in (0, 2)$, then set $u(z) = \exp\pr{- \abs{z}^{1 - \frac N 2}}$.
A computation shows that $u$ satisfies \eqref{ellipEq} where
$$V(z) := \pr{1 - \frac N 2}^2 \pr{\abs{z}^{1 - \frac N 2} - 1}\abs{z}^{-1 - \frac N 2}$$ 
satisfies $\abs{V(z)} \lesssim \abs{z}^{-N}$.
On the other hand, for any $\be > 0$, with $u(z) = \abs{z}^{-\be}$ on $\abs{z} > 1$, we see that $\disp \LP u = \be^2 \abs{z}^{-\be-2}$ and therefore $u$ satisfies \eqref{ellipEq} on an exterior domain with $V(z) := \be^2\abs{z}^{-2}$.
In particular, we may not have exponential behavior when $V$ decays fast enough, which explains why we restrict ourselves to $N < 2$.

To prove Theorem \ref{mainThm}, we use an iterative argument that is reminiscent of the one in \cite{Dav14}, see also \cite{LW14, DKW19, Dav20a, Dav25}.
To initialize the iteration, we apply a quantitative estimate of the form \eqref{est} with $\be = 1$.
This result, which verifies Landis' conjecture in the real-valued planar setting, was originally proved by Logunov, Malinnikova, Nadirashvili, and Nazarov in \cite{LMNN20}, and we formulate it in Theorem \ref{LandisGrowth} below.
The iteration argument then relies on repeated applications of Proposition \ref{InductiveProp} which is proved using the ideas from \cite{LMNN20}.
Roughly speaking, Proposition \ref{InductiveProp} shows that if an estimate like \eqref{est} holds with $\be = \be_0$, then for some $x_1$ with $\abs{x_1} \gg \abs{x_0}$, another estimate like \eqref{est} holds with $x_0$ replaced by $x_1$ and $\be = \be_1 \in \brac{1, \be_0}$.
When $\be_0 = 1$, Proposition \ref{InductiveProp} isn't useful, but when $\be_0 > 1$, we can decrease the exponent, i.e., make $\be_1 < \be_0$.
Therefore, to benefit from the iteration argument, we need to transform to a situation where $\be_0 > 1$.
We observe that if $u$ is composed with the real-variable version of the conformal transformation $z \mapsto z^\al$, then the new function also satisfies a Schr\"odinger equation.
By choosing $\al > 1$ appropriately, we can ensure that the new potential function is bounded and that the new solution function satisfies a version of \eqref{est} with $\be > 1$.
By repeatedly applying Proposition \ref{InductiveProp} to the transformed equation, we can make $\be$ arbitrarily close to $1$.
Finally, to reach the conclusion, we undo the change of variables.

We use the notation $B_r(z)$ to denote a ball of radius $r > 0$ centered at the point $z$, abbreviated by $B_r$ when the center is clear.
Generic constants are denoted by $c, C$ and may change from line to line without comment.
Specific constants will be indicated by subscripts.

The article is organized as follows.
In Section \ref{harmonic}, we present a unique continuation theorem for harmonic functions in punctured domains.
The content of this section is very similar to \cite[Section 5]{LMNN20} and \cite[Section 2]{Dav24}.
The iterative result described by Proposition \ref{InductiveProp} is the content of Section \ref{localProof}.
Proposition \ref{InductiveProp} is a three-ball inequality for solutions to Schr\"odinger equations and its proof relies on the results from Section \ref{harmonic}. % that apply to harmonic functions. 
In Section \ref{TransMaps}, we introduce the real-valued versions of $z \mapsto z^\al$ and record some of their properties.
In particular, we show how solutions behave when they are composed with these transformations.
Finally, the proof of Theorem \ref{mainThm} is presented in Section \ref{MainProof}.

\section{Decay properties of harmonic functions in punctured domains}
\label{harmonic}

In this section, we present and prove quantitative unique continuation results (in the form of three-ball inequalities) for harmonic functions in punctured domains.
We recall the following application of the Harnack inequality which appears in \cite[Claim 5.2]{LMNN20} and is repeated in \cite[Lemma 2.1]{Dav24}.

\begin{lem}[Harnack application]
\label{discBounds}
Let $\set{D_j}$ be a finite collection of $100$-separated unit disks in the plane.
Assume that $h$ is real-valued and harmonic in $\R^2 \setminus \bigcup D_j$ and that for each index $j$, $h$ doesn't change sign in $5D_j \setminus D_j$.
There exists an absolute constant $C_H \ge 10$ for which
\begin{enumerate}
\item $\disp \max_{\del \pr{3 D_j}} \abs{h} \le C_H \min_{\del \pr{3 D_j}} \abs{h}$
\item $\disp \max_{\del \pr{3 D_j}} \abs{\gr h} \le C_H \min_{\del \pr{3 D_j}} \abs{h}$.
\end{enumerate}
\end{lem}

\begin{proof}
An application of the Harnack inequality shows that there exists $C_H > 0$ so that for every $j$
\begin{align*}
\max_{\del \pr{3 D_j}} \abs{h}
&\le \sup_{4 D_j \setminus 2 D_j} \abs{h} 
\le C_H \inf_{4 D_j \setminus 2 D_j} \abs{h}
\le C_H \min_{\del \pr{3 D_j}} \abs{h}.
\end{align*}
For each $z \in \del \pr{3 D_j}$, since $h$ doesn't change signs in $B_2(z)$, an application of Cauchy's inequality as in \cite[Lemma 1.11]{HL11} shows that $\abs{\gr h(z)} \le \abs{h(z)}$ and the conclusion follows.
\end{proof}

Now we state and prove the main result of this section.
The following is a slight modification of \cite[Lemma 2.1]{Dav24}, which resembles the result \cite[Theorem 5.3]{LMNN20}.

\begin{prop}[Three-ball inequality for harmonic function in punctured domain]
\label{harmonicProp}
Let $\set{D_j}$ be a finite collection of $100$-separated unit disks in the plane for which $0 \notin \bigcup 3 D_j$.
For some $R \ge 2^{10}$, let $h$ be a harmonic function in $B_R \setminus \bigcup D_j$ with the property that for each index $j$, $h$ doesn't change sign in $\pr{5D_j \setminus D_j} \cap B_R$.
Assume that for some $S \in \pr{2^8, \frac R 2}$ and some $M > \log\pr{16R}$, with $T := R - \frac S {32}$, 
it holds that
\begin{equation}
\label{normalization}
\sup_{B_T \setminus \bigcup 3 D_j} \abs{h}  \ge e^{-M} \sup_{B_R \setminus \bigcup 3 D_j} \abs{h}.
\end{equation}
Then for every $r \in \pr{0, \frac R {2^{10}}}$, we have
\begin{equation}
\label{lowerBound}
\sup_{B_r \setminus \bigcup 3 D_j} \abs{h} 
\ge  \pr{\frac{16 r} R}^{\kappa} \sup_{B_T \setminus \bigcup 3 D_j} \abs{h},
\end{equation}
where $\kappa(R,S, M) = \max \set{6C_HR, 2^{13} MRS^{-1}}$ and $C_H \ge 10$ is from Lemma \ref{discBounds}. \\
\end{prop}

\begin{rem}
In contrast to \cite[Lemma 2.1]{Dav24}, we no longer assume that $S$ is bounded from below with respect to $R$ and we impose a lower bound on $M$.
This flexibility is useful for the iterative argument.
\end{rem}

\begin{proof}
We may assume without loss of generality that 
\begin{equation}
\label{normalizationAssump}
\sup_{B_T \setminus \bigcup 3 D_j} \abs{h} = 1.
\end{equation}
Set $k = \max \set{2C_HR, 2^{10} M R S^{-1}}$. 
For the sake of contradiction, assume that 
\begin{equation}
\label{contraBound}
\sup_{B_r \setminus \bigcup 3 D_j} \abs{h} \le \pr{\frac {16r} R}^{3k}.
\end{equation}

Define the punctured annular region
$$\Om := \set{\frac r 2 < \abs{z} < R -1} \setminus \bigcup 3 D_j$$
and the function
$$f(z) = \frac{h_x - i h_y}{z^k}.$$
Observe that $f$ is analytic in $\Om$ and $\abs{f(z)} = \abs{\gr h(z)} \abs{z}^{-k}$.
We'll analyze the behavior of $f$ over $\Om$.
We begin with bounding $h$ and $\gr h$ over the innermost and outermost parts of the boundary of $\Om$.

Let $W_1$ be the connected component of $\del \Om$ that intersects the inner circle $\set{\abs{z} = \frac r 2}$.
If $z \in W_1$, then there are three cases to consider:
\begin{itemize}
\item[(a)] $\abs{z} \ne \frac r 2$.
\item[(b)] $\abs{z} = \frac r 2$ and there exists $j$ for which $z \in 4 D_j \setminus 3 D_j$.
\item[(c)] $\abs{z} = \frac r 2$ and $z \cap 4 D_j$ is empty for all $j$.
\end{itemize}

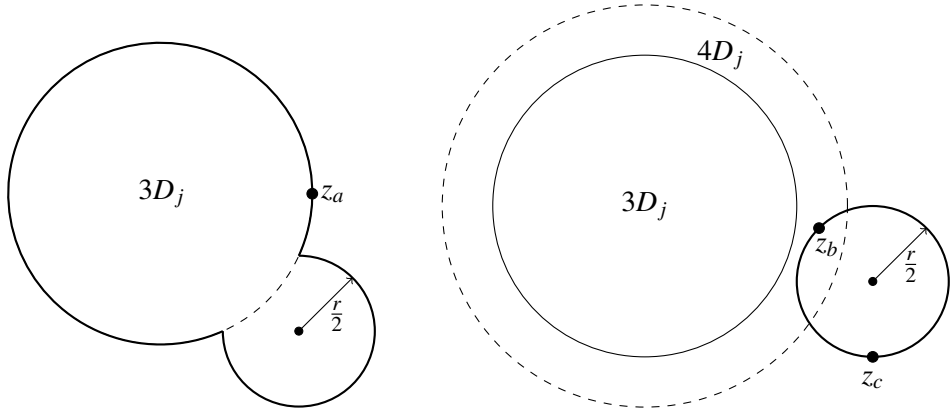
\begin{figure}[h]
\label{W1Figure}
\begin{tikzpicture}
\draw[thick] (-1,0) arc (180:450:1cm);
\draw[thick] (0,1) arc (-24.3:294.3:2cm);
\draw[dashed] (0,1) arc (-24.3:360:2cm);
\draw [fill=black] (0,0) circle (1.5pt);
\draw[->] (0,0) -- (0.707, 0.707);
\draw[color=black] (0.5,0.2) node {$\frac r 2$};
\draw [fill=black] (0.177,1.822) circle (2pt);
\draw[color=black] (-1.8,1.822) node {$3D_j$};
\draw[color=black] (0.45,1.822) node {$z_a$};
\end{tikzpicture}
\qquad 
\begin{tikzpicture}
\draw[thick] (-1,0) arc (180:540:1cm);
\draw (-1,1) arc (0:360:2cm);
\draw[dashed] (-0.333,1) arc (0:360:2.666cm);
\draw[color=black] (-3,1) node {$3D_j$};
\draw[color=black] (-2,3) node {$4D_j$};
\draw [fill=black] (0,0) circle (1.5pt);
\draw[->] (0,0) -- (0.707, 0.707);
\draw[color=black] (0.5,0.2) node {$\frac r 2$};
\draw [fill=black] (-0.707, 0.707) circle (2pt);
\draw[color=black] (-0.6,0.45) node {$z_b$};
\draw [fill=black] (0, -1) circle (2pt);
\draw[color=black] (0,-1.3) node {$z_c$};
\end{tikzpicture}
\caption{Possible images of $W_1$ with cases (a), (b) and (c) illustrated by the points $z_a$, $z_b$, and $z_c$, respectively.}
\label{projPics}
\end{figure}

\nid \textbf{Case (a):} There exists $j$ for which $z \in \del \pr{3 D_j}$ and $3D_j \cap \set{\abs{z} = \frac r 2}$ is non-empty.
An application of Lemma \ref{discBounds} combined with the fact that $\del \pr{3D_j} \cap B_R$ is non-empty shows that
\begin{align*}
\abs{h(z)}, \abs{\gr h(z)} 
&\le C_H \min_{\del \pr{3 D_j}} \abs{h} 
\le C_H \sup_{B_r \setminus \bigcup 3 D_j} \abs{h}
\le C_H \pr{\frac {16r} R}^{3k},
\end{align*}
where the last inequality follows from \eqref{contraBound}.
\\
\textbf{Case (b):} Since $h$ doesn't change signs in $B_1(z)$, then an application of \cite[Lemma 1.11]{HL11} shows that
\begin{align*}
\abs{\gr h(z)} 
&\le 2 \abs{h(z)}
\le 2 \sup_{B_R \setminus \bigcup 3 D_j} \abs{h}
\le 2  \pr{\frac {16r} R}^{3k},
\end{align*}
where the second inequality uses that $z \notin \bigcup 3 D_j$ and we have again applied \eqref{contraBound}.
\\
\textbf{Case (c):} Let $d = \min\set{1, \frac r 2}$ and observe that $B_d\pr{z} \subset B_R \setminus \bigcup 3 D_j$, so an application of Cauchy's inequality, \cite[Lemma 1.10]{HL11}, shows that
\begin{align*}
\abs{\gr h(z)} 
&\le \frac{2}{d} \sup_{B_d\pr{z}} \abs{h}
\le \frac{2}{d} \sup_{B_R \setminus \bigcup 3 D_j} \abs{h}
\le \frac 2 d  \pr{\frac {16r} R}^{3k}.
\end{align*}
If $d = 1$, since $k > 1$, then $\disp \frac 2 d \pr{\frac{16 r} R}^k =  2 \pr{\frac{16 r} R}^k < \frac{32r}{R} < \frac 1 {2^5} < \frac 1 2$.
On the other hand, if $d = \frac r 2$, then $\disp \frac 2 d \pr{\frac{16 r} R}^k = \frac 4 r \pr{\frac{16 r} R}^k = \frac{64}{R} \pr{\frac{16 r} R}^{k-1} < \frac 1 {2^4} < \frac 1 2$.

Since $k \ge 2C_H R \ge 2^{11} C_H$, then $2^{10k} \ge \max\set{C_H, 2} = C_H$ and $\pr{\frac r R}^k \le 2^{-10k} \le \frac 1 {\max\set{C_H, 2}}$.
Therefore, by combining all three cases, we see that
\begin{align}
\label{W1Bound}
\sup_{W_1} \abs{h}, \; \sup_{W_1} \abs{\gr h} \le \pr{\frac {16r} R}^{2k}.
\end{align}

Let $W_2$ be the connected component of $\del \Om$ that intersects the outer circle $\set{\abs{z} = R-1}$ and note that $W_2 \su \overline{B_{R-1}} \setminus B_{R-7}$.
Now if $z \in W_2$, there are two cases to consider:
\begin{itemize}
\item[(a)] there exists $j$ for which $z \in 4 D_j$.
\item[(b)] $\abs{z} = R-1$ and $z \cap 4 D_j$ is empty for all $j$.
\end{itemize}

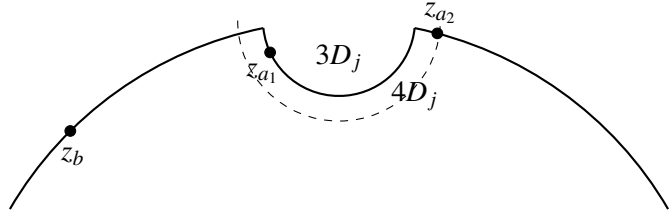
\begin{figure}[h]
\label{W1Figure}
\begin{tikzpicture}
\draw[thick] (4.33,2.5) arc (30:78.52:5cm);
\draw[thick] (-4.33,2.5) arc (150:101.48:5cm);
\draw[thick] (-0.995,4.9) arc (185.739:354.261:1cm);
\draw[dashed] (-1.333,5) arc (180:360:1.333cm);
\draw[color=black] (0,4.5) node {$3D_j$};
\draw[color=black] (1,4) node {$4D_j$};
\draw [fill=black] (1.294, 4.8296) circle (2pt);
\draw[color=black] (1.34, 5.1) node {$z_{a_2}$};
\draw [fill=black] (-0.9063, 4.577) circle (2pt);
\draw[color=black] (-1, 4.3) node {$z_{a_1}$};
\draw [fill=black] (-3.5355, 3.5355) circle (2pt);
\draw[color=black] (-3.5,3.2) node {$z_b$};
\end{tikzpicture}
\caption{A possible image of $W_2$ with case (a) illustrated by the points $z_{a_1}$ and $z_{a_2}$, and case (b) illustrated by $z_b$.}
\label{projPics}
\end{figure}

\nid\textbf{Case (a):} Since $h$ doesn't change sign in $B_1(z)$, then an application of \cite[Lemma 1.10]{HL11} shows that
\begin{align*}
\abs{\gr h(z)} 
&\le 2 \abs{h(z)}
\le 2 \sup_{B_R \setminus \bigcup 3 D_j} \abs{h}
\le 2 e^{M},
\end{align*}
where we have applied \eqref{normalization}.
\\
\textbf{Case (b):} Since $B_1(z) \su B_R \setminus \bigcup 3 D_j$, then
\begin{align*}
\abs{\gr h(z)} 
&\le 2 \sup_{B_1\pr{z}} \abs{h}
\le 2 \sup_{B_R \setminus \bigcup 3 D_j} \abs{h}
\le 2 e^M.
\end{align*}

By combining both cases, we see that
\begin{equation}
\label{W2Bound}
\sup_{W_2} \abs{h} \le e^M, \; \sup_{W_2} \abs{\gr h} \le 2 e^M.
\end{equation}

We now use these estimates on $h$ from above to understand the behavior of the analytic function $\disp f = \frac{h_x - i h_y}{z^k}$.
Define the set $\Om_1 \su \Om$ as
$$\Om_1 :=  \set{\frac r 2 < \abs{z} < T} \setminus \bigcup 3 D_j.$$
Using containment, assumption \eqref{contraBound}, and our normalization \eqref{normalizationAssump}, we see that
\begin{align*}
\sup_{B_{\frac r 2} \setminus \bigcup 3D_j} \abs{h}
\le \sup_{B_r \setminus \bigcup 3D_j} \abs{h}
\le \pr{\frac {16 r} R}^{3k}
< 1 
= \sup_{B_T \setminus \bigcup 3 D_j} \abs{h}.
\end{align*}
Therefore, there exists $z_0 \in \Om_1$ for which $\abs{h(z_0)} = 1$.
By \eqref{W1Bound}, we have 
\begin{align*}
\sup_{W_1} \abs{h} \le \pr{\frac {16r} R}^{2k} < \frac 1 2,
\end{align*}
so there exists $z_1 \in W_1$ for which $\abs{h(z_1)} = \al < \frac 1 2$.
Let $\Ga$ be a path in $\Om_1$ from $z_0$ to $z_1$ for which $\ell\pr{\Ga} \le 4 R$.
If we assume that $\abs{\gr h(z)}  < \frac 1 {8R}$ for all $z \in \Ga$, then 
\begin{align*}
\frac 1 2 
< \abs{h(z_1) - h(z_0)} 
= \abs{\int_\Ga \gr h(w) \cdot dw}
\le \int_\Ga \abs{\gr h(w)} \abs{dw}
< \frac 1 {8R} \ell\pr{\Ga}
< \frac 1 2,
\end{align*}
which is impossible so it follows that
\begin{align*}
\sup_{\Om} \abs{\gr h}
\ge \sup_{\Om_1} \abs{\gr h}
\ge \frac 1 {8R}.
\end{align*}
Therefore,
\begin{align}
\label{OmSup}
\sup_{\Om} \abs{f}
\ge \sup_{\Om_1} \abs{f}
\ge \sup_{\Om_1} \abs{\gr h} T^{-k}
\ge \frac 1 {8R} T^{-k}.
\end{align}
An application of \eqref{W1Bound} shows that
\begin{align}
\label{W1Sup}
\max_{W_1} \abs{f} 
&\le \max_{W_1} \abs{\gr h} \pr{\frac 2 r}^k
\le \pr{\frac {16r} R}^{2k} \pr{\frac 2 r}^k
= \pr{\frac {2^9 r}{R}}^k R^{-k}
< 2^{-k} R^{-k},
\end{align}
where we have used that $\frac R r > 2^{10}$.
Since $k \ge 2 C_H R \ge 20 R$, then $2^{-k} < \frac 1 {8R}$.
In particular, by combining \eqref{W1Sup} with \eqref{OmSup}, we deduce that
\begin{align*}
\max_{W_1} \abs{f} 
< \frac 1 {8R} T^{-k}
\le \sup_{\Om} \abs{f}.
\end{align*}
An application of \eqref{W2Bound} shows that
\begin{align}
\label{W2Sup}
\max_{W_2} \abs{f} 
&\le \max_{W_2} \abs{\gr h} \pr{R - 7}^{-k}
\le 2 e^M \pr{\frac{R - \frac{S}{32}}{R - 7}}^{k} T^{-k} .
\end{align}
Now
\begin{align}
\label{RBig2}
2 e^M \pr{\frac{R - \frac{S}{32}}{R - 7}}^{k} < \frac 1 {8R} 
&\iff k \log \pr{\frac{R - 7 }{R - \frac S{32}}} > M + \log\pr{16R}.
\end{align}
Since $S \ge 2^8$, then $\frac S {32} - 7 \ge \frac S {2^8}$ and we see that
\begin{align*}
\log \pr{\frac{R - 7 }{R - \frac S{32}}}
&= \log \pr{1 + \frac{\frac{S}{32} - 7 }{R - \frac S{32}}}
\ge \log \pr{1 + \frac {S}{2^8 R}}
\ge \frac{S}{2^8 R} \pr{1 - \frac{S}{2^9 R}}.
\end{align*}
As $S \le \frac R 2$, then $1 - \frac{S}{2^9 R} \ge 1 - 2^{-10} > \frac 1 2$ and we deduce that
$$\log \pr{\frac{R - 7 }{R - \frac S{32}}} > \frac{S}{2^9 R}.$$
Since $\disp k \ge 2^{10} M R S^{-1}$, then
$$\frac k 2 \log \pr{\frac{R - 7 }{R - \frac S{32}}} > M > \log\pr{16R},$$
where the second inequality is by assumption. %, and we see that
Therefore, \eqref{RBig2} holds.
Combining \eqref{RBig2} with \eqref{W2Sup} and \eqref{OmSup} shows that
\begin{align*}
\max_{W_2} \abs{f} 
< \sup_{\Om} \abs{f}.
\end{align*}
Since $f$ is a holomorphic function in $\Om$, then the maximum principle guarantees that $\disp \sup_{\Om} \abs{f} = \sup_{\del \Om} \abs{f}$.
As shown above, the maximum doesn't occur on $W_1$ or $W_2$, so there must exist a disk $3 D_j \su \set{\frac r 2 < \abs{z} < R-1}$ for which $\disp \sup_{\Om} \abs{f} = \sup_{\del \pr{3D_j}} \abs{f}$.

Considering only the disks $D_j$ for which $3 D_j \su \set{\frac r 2 < \abs{z} < R-1}$, define $z_j \in \del \pr{3 D_j}$ to be the point that is closest to the origin, i.e. has the smallest modulus.
Then set $\disp m_j = \min_{\del \pr{3 D_j}} \abs{h}$.
Define $j_0$ to be the index for which
\begin{equation}
\label{j0Defn}
m_{j_0} \abs{z_{j_0}}^{-k} = \max_j m_{j} \abs{z_{j}}^{-k}
\end{equation}
and let $j_1$ be the index for which 
$$\sup_{\Om} \abs{f} = \sup_{\del \pr{3 D_{j_1}}} \abs{f}.$$
For any $z \in \Om$, an application of Lemma \ref{discBounds} shows that
\begin{equation}
\label{grhBound}
\abs{\gr h(z)} \abs{z}^{-k} 
\le \sup_{\Om} \abs{f}
= \sup_{\del \pr{3 D_{j_1}}} \abs{\gr h(z)} \abs{z_{j_1}}^{-k} 
\le C_H m_{j_1} \abs{z_{j_1}}^{-k} 
\le C_H m_{j_0} \abs{z_{j_0}}^{-k},
\end{equation}
so we see that 
\begin{equation*}
\abs{\gr h(z)} 
\le C_H m_{j_0} \pr{\frac{\abs{z}}{\abs{z_{j_0}}}}^{k}.
\end{equation*}
Since $\disp \sup_{\Om} \abs{\gr h} \ge \frac 1 {8R}$ and this estimate holds for every $z \in \Om$, then
$$\frac 1 {8R} \le C_H m_{j_0} \pr{\frac{\abs{z}}{\abs{z_{j_0}}}}^{k}\le C_H m_{j_0} \pr{\frac{2R}{r}}^{k}.$$
In particular,
\begin{equation}
\label{hzj0Bound}
\abs{h(z_{j_0})} \ge m_{j_0} \ge \frac 1 {8C_HR} \pr{\frac{r}{2R}}^{k}.
\end{equation}
Define $s = \inf\set{\tau \le 1 : t z_{j_0} \in \Om \text{ for all } t \in \pr{\tau,1}}$ so that the straight line path given by $\ga(t) = t z_{j_0}$ for $s < t < 1$ is contained in $\Om$ while $s z_{t_0} \in \del \Om$.
Integrating $\gr h$ along $\ga$ shows that
\begin{align*}
h(z_{j_0}) - h(s z_{j_0})
&= \int_\ga \gr h\pr{z} \cdot d z
= \int_{s}^1 \gr h(t z_{j_0}) \cdot z_{j_0} dt.
\end{align*}
Applications of \eqref{grhBound} and \eqref{hzj0Bound} show that 
\begin{equation}
\begin{aligned}
\label{integralBound}
\abs{h(s z_{j_0})}
&\ge \abs{h(z_{j_0})} - \abs{\int_{s}^1 \gr h(t z_{j_0}) \cdot z_{j_0} dt}
\ge \abs{h(z_{j_0})} - \abs{z_{j_0} } \abs{\int_{s}^1 C_H m_{j_0}t^{k} dt} \\
&\ge m_{j_0} - \frac{m_{j_0} C_H R }{k+1} 
= m_{j_0} \pr{1 - \frac{C_H R}{k+1}}
> \frac{m_{j_0}}{2},
\end{aligned}
\end{equation}
where the last inequality uses that $k+1 > 2C_HR$.
Since $k \ge 2C_H R$, then $2^{k} > 16 C_H R$ and it follows that 
$$\pr{\frac R r}^k \ge 2^{10k} > 16 C_H R \cdot 2^{9k}.$$
Combining \eqref{integralBound} with \eqref{hzj0Bound} shows that
$$\abs{h(s z_{j_0})} > \frac{m_{j_0}}{2} \ge \frac 1 {16 C_H R} \pr{\frac{r}{2R}}^{k} > 2^{9k} \pr{\frac r R}^k\pr{\frac{r}{2R}}^{k} = \pr{\frac{16 r} R}^{2k}.$$
By comparing this bound with \eqref{W1Bound}, we conclude that $s z_{j_0} \notin W_1$ so it must hold that $s z_{j_0} \in \del \pr{3 D_{j_2}}$ for some $3 D_{j_2} \su \set{\frac r 2 < \abs{z} < R-1}$.
Then Lemma \ref{discBounds}, that $\abs{z_{j_2}} \le \abs{s z_{j_0}}$, and \eqref{integralBound} show that 
\begin{align}
\label{j2j0Comp}
C_Hm_{j_2} \abs{z_{j_2}}^{-k} 
&\ge \sup_{\del \pr{3 D_{j_2}}} \abs{h} \abs{z_{j_2}}^{-k}
\ge \abs{h(s z_{j_0})} \abs{s z_{j_0}}^{-k}
\ge \frac{m_{j_0}}{2} s^{-k} \abs{z_{j_0}}^{-k}.
\end{align}
Since $z_{j_0} \in \del \pr{3 D_{j_0}}$ and $s z_{j_0} \in \del \pr{3 D_{j_2}}$ where $j_0 \ne j_2$, and the balls $\set{D_j}$ are of unit radius and $100$-separated, then $\abs{z_{j_0} - s z_{j_0}} \ge 96$.
After rearrangement, we see that $s^{-k} \ge \pr{1 - \frac{96}{R}}^{-k}$.
Since $10 \le C_H$, $2C_HR \le k$, and $\frac{96}{R} < - \log\pr{1 - \frac{96}{R}}$, then 
\begin{align*}
\log\pr{2C_H}
&< 96 \cdot 2 C_H
\le \frac{96}{R} k
< - k \log\pr{1 - \frac{96}{R}}
\le - k \log s,
\end{align*}
from which it follows that $s^{-k} > 2C_H$.
We then conclude from \eqref{j2j0Comp} that $m_{j_2} \abs{z_{j_2}}^{-k} > m_{j_0} \abs{z_{j_0}}^{-k}$ which contradicts \eqref{j0Defn} and gives the desired contradiction.
In other words, \eqref{contraBound} fails to hold and we see that
\begin{align*}
\sup_{B_r \setminus \bigcup 3 D_j} \abs{h} 
&> \pr{\frac{16r} R}^{3k}
= \pr{\frac {16r} R}^{3k} \sup_{B_T \setminus \bigcup 3 D_j} \abs{h},
\end{align*}
which implies \eqref{lowerBound} by our choice of $k$.
\end{proof}

\section{The iterative proposition}
\label{localProof}

Here we present the proposition that is used repeatedly in the iteration argument.
This unique continuation result takes the form of a three-ball inequality for solutions to Schr\"odinger equations.
The techniques used to prove this theorem are very similar to those that appear in the proofs of \cite[Theorem 2.2]{LMNN20} and \cite[Theorem 1.2]{Dav24}.

\begin{prop}[Iterative Proposition]
\label{InductiveProp}
Given $a > 0$ and $R \ge \hat R_0(a)$, let $W : B_R \to \R$ satisfy $\norm{W}_{L^\iny(B_R)} \le a^2$ and let $v : B_R \to \R$ be a solution to 
$$-\LP v + W v = 0 \; \text{ in } B_R$$
with
\begin{equation}
\label{propUB}
\norm{v}_{L^\iny(B_R)} \le e^{c_1 R}
\end{equation}
for some $c_1 \ge 1$.
Assume that for some $S \in \pr{2^8, \frac R 2}$, there exists $z_0 \in \overline{B}_{R-S}$ and $L \ge 0$ such that
\begin{equation}
\label{propLB}
\abs{v(z_0)} \ge e^{- L}.
\end{equation}
Then there exists $r_0(a, R) > 1$ so that whenever $r \in (0, r_0)$, it holds that
\begin{align}
\label{propConc}
\norm{v}_{L^\iny(B_r)}
&\ge \pr{\frac{r} {R}}^{\tau},
\end{align}
with
$$\tau(R, S, L, a, c_1) = \max \set{3 a C_H\sqrt{\frac{C_K }{32\ln 2}} R \sqrt{\log R}, 2^{14} \pr{c_1 R + L + \frac {c_d}{\log R}} RS^{-1}} +\frac 1 5 \pr{L + \frac{c_d}{\log R}},$$
where $C_K$ and $c_d$ are universal constants, and $C_H$ is the Harnack constant from Lemma \ref{discBounds}.
\end{prop}

\begin{proof}
Let $v : B_R \to \R$, $W : B_R \to \R$ and $z_0 \in \overline{B}_{R - S}$ be as given in the statement. 
Let $F_0$ denote the nodal set of $v$, i.e. 
$$F_0 = \set{z \in \R^2 : v(z) = 0}.$$
Set $c_s = 2^{10} 10^2$.
For $\rho > 0$ to be specified below, there exists a set $F_1 \su B_R$ that consists of a collection of $c_s \rho$-separated closed disks of radius $\rho$ that are also $c_s \rho$-separated from $0$, $z_0$, $F_0$, and $\del B_R$.
Moreover, the set $F_0 \bigcup F_1 \bigcup \del B_R$ is a $10 c_s \rho$-net in $B_R$.
A more detailed description of this process is given in \cite[\S 2, Act I]{LMNN20}.

Define $\Om = B_R \setminus \pr{F_0 \bigcup F_1}$ and $\Om_1 = B_R \setminus F_1$.
As shown in \cite[\S 3.1]{LMNN20}, there exists a universal constant $c_P$ (that depends on $c_s$) so that $\Om$ has Poincar\'e constant bounded above by $c_P \rho^2$.
In particular, since $c_P \rho^2 \norm{W}_{L^\iny(B_R)} \le c_P \rho^2 a^2$, then by choosing $\rho \le \rho_0$, a universal constant, we can apply \cite[Lemma 3.2]{LMNN20}.
For $\eps \ll 1$ to be defined later on, let 
\begin{equation}
\label{rhoDef}
\rho = \eps a^{-1}.
\end{equation}
An application of the arguments in \cite[\S 3.2]{LMNN20} then shows that there exists $\phi : \Om \to \R$ with the properties that
\begin{align}
&\LP \phi - W \phi = 0 \text{ in } \Om
\nonumber \\
&\phi -1 \in W^{1,2}_0(\Om)
\nonumber \\
&\norm{\phi -1}_\iny \le c_b \pr{\rho a}^2 = c_b \eps^2,
\label{vpuBound}
\end{align}
where $c_b$ is a universal constant that depends on $c_P$, and we have used \eqref{rhoDef}.
By extending $\phi$ to equal $1$ across $F_0 \bigcup F_1$, it is then shown in \cite[Lemma 4.1]{LMNN20} that $\disp f := \frac v \phi \in W^{1,2}_{\loc}(B_R)$ is a weak solution to the divergence-form equation
$$\di\pr{\phi^2 \gr f} = 0 \; \text{ in } \Om_1.$$
Moreover, the bound in \eqref{vpuBound} implies that for any $z \in B_R$,
\begin{equation}
\label{ufComparison}
\pr{1 - c_b \eps^2}  \abs{f(z)}  \le \abs{v(z)} \le \pr{1 + c_b \eps^2}  \abs{f(z)}.
\end{equation}

We introduce the Beltrami coefficient $\mu$, defined as follows:
\begin{equation*}
\mu = \left\{ \begin{array}{ll}
\frac{1 - \phi^2}{1 + \phi^2} \frac{f_x + i f_y}{f_x - i f_y} & \text{ in } \Om_1 \text{ when } \, \gr f \ne 0 \\
0 & \text{ otherwise}
\end{array} \right..
\end{equation*}
Since $\abs{\mu} \lesssim \eps^2$, then as shown in \cite{AIM09}, there exists a $K$-quasiconformal homeomorphism of the complex plane where $K \le 1 + C_K \eps^2$ and $C_K$ depends on $c_b$.
That is, there exists some  $w \in W^{1,2}_{\loc}$ that satisfies the Beltrami equation $\disp \frac{\del w}{\del \overline{z}} = \mu \frac{\del w}{\del z}$. 
In fact, an application of the Riemann uniformization theorem shows that there exists a $K$-quasiconformal homeomorphism $g: B_R \to B_R$ that is onto with $g(0) = 0$.
Moreover, the function $h : = f \circ g^{-1}$ is harmonic in $g(\Om_1)$.

Mori's Theorem implies that
\begin{equation}
\label{MoriThm}
\frac 1 {16} \abs{\frac{z_1 - z_2}{R}}^K
\le \frac{\abs{g(z_1) - g(z_2)}}{R}
\le 16 \abs{\frac{z_1 - z_2}{R}}^{\frac1 K}.
\end{equation}
Thus, if we set 
\begin{equation}
\label{epsDef}
\eps = \sqrt{\frac{\ln 2}{2 C_K \log R}},
\end{equation}
then $\disp K \in \brac{1, 1 + \frac{\ln 2}{2\log R}}$ and $R \simeq R^K \simeq R^{\frac 1 K}$.
Since the arguments above hold when $\rho = \frac \eps a \le \rho_0$, then we ensure that $R \ge R_1 := \exp\brac{\frac{\ln 2}{2 C_K \pr{a \rho_0}^2 }}$.
If $\abs{z_1 - z_2} = \rho$, \eqref{MoriThm} shows that
\begin{align*}
\abs{g(z_1) - g(z_2)}
\le 16 R \pr{\frac \rho R}^{\frac 1 K}
= 16 \rho \pr{\frac {aR} \eps}^{1 - \frac 1 K}.
\end{align*}
If $R \ge R_2 := \inf\set{R > 0 : \frac{R^2}{\log R}  \ge \frac{2 C_K a^2 }{\ln 2}}$, then $\frac{a}{\eps} \le R$ and
\begin{align*}
\abs{g(z_1) - g(z_2)}
\le 16 \rho \exp\brac{2\log R\pr{1 - \frac 1 K }}
\le 32 \rho.
\end{align*}
On the other hand, if $\abs{z_1 - z_2} = c_s \rho$, \eqref{MoriThm} implies that
\begin{align*}
\abs{g(z_1) - g(z_2)}
\ge \frac R {16}  \pr{\frac {c_s \rho} R}^{K}
\ge c_s \frac \rho {16}  \pr{\frac {\eps} {aR}}^{K- 1}
\ge c_s \frac \rho {16} \exp\brac{-2 \log R\pr{K-1}}
\ge c_s \frac \rho {32}.
\end{align*}
Recalling that $c_s = 2^{10} 10^2$, we may conclude that
$$\abs{g(z_1) - g(z_2)} \ge 3200 \rho.$$
Therefore, $h$ is harmonic in $B_R \setminus \bigcup D_j$, where each $D_j$ is a disk of radius $32\rho$.
Moreover, the disks are $3200 \rho$-separated from each other, $0$, and $g(z_0)$, while $h$ doesn't change sign in any of the annuli $100 D_j \setminus D_j$.
Since $z_0 \in \overline{B}_{R - S} \cap \Om$, then for every $z \in \del B_R$, $\abs{z_0 - z} \ge S \ge 2^8$ and the distortion estimate described by \eqref{MoriThm} shows that
\begin{align*}
\abs{g(z_0) - g(z)} 
&\ge \frac R {16} \abs{\frac{z_0 - z}{R}}^{K}
\ge \frac R {16} \pr{\frac{S}{R}}^{K} 
= \frac S {16} \pr{\frac{2^8}{R}}^{K-1} 
\ge \frac S {16} R^{1-K}
\ge \frac S {16} \exp\pr{- \frac{\ln 2}{2}}
> \frac S {32} .
\end{align*}
It follows that $g(z_0) \in B_{T} \setminus \bigcup 3 D_j $, where $T := R - \frac{S}{32}$.

Since $g : B_R \to B_R$, then we may rescale the map to get 
$$\tilde g := \frac g {32\rho} : B_R \to B_{\frac R {32\rho}}$$ 
which is onto with $\tilde g(0) = 0$.
Using \eqref{rhoDef} and \eqref{epsDef}, set $C_1 =  \frac {a} {32} \sqrt{\frac{2 C_K }{\ln 2}}$ and define
\begin{equation*}
\WT R 
= \frac R {32 \rho} 
= C_1 R \sqrt{\log R}.
\end{equation*}
From here, we see that $\tilde h := f\circ \tilde g^{-1}$ is harmonic in $\tilde g\pr{\Om_1}$.
In particular, $\tilde h$ is harmonic in $B_{\WT R} \setminus \bigcup \WT D_j$, where the $\WT D_j$ are unit disks that are 100-separated from each other, from $0$, and from $\tilde g(z_0)$.
Moreover, $\tilde h$ doesn't change signs on any annuli $5 \WT D_j \setminus \WT D_j$.
With $\WT S := C_1 S \sqrt{\log R}$ and $\WT T := \WT R - \frac{\WT S}{32} = C_1 T \sqrt{\log R}$, $\tilde g(z_0) \in B_{\WT T} \setminus \bigcup 3 \WT D_j$.
Set $\hat R_0 = \max \set{R_1, R_2, 2^{10},  \exp\pr{C_1^{-2}}}$ so the above arguments hold.
Since $R \ge \hat R_0$ implies that $C_1 \sqrt{\log R} \ge 1$ and $R \ge 2^{10}$, then $\WT R = C_1 R \sqrt{\log R} \ge 2^{10}$ as well.
Because $S \in \pr{2^8, \frac R 2}$, then $\WT S \in \pr{2^8, \frac{\WT R}{2}}$.

As $z_0$ satisfies the bound in \eqref{propLB}, then \eqref{ufComparison} shows that
\begin{align*}
\pr{1 + c_b \eps^2} \abs{f(z_0)}
\ge \abs{v(z_0)}
\ge \exp\pr{- L}.
\end{align*}
Since $\tilde g(z_0) \in B_{\WT T} \setminus \bigcup 3 \WT D_j$, then
\begin{align}
\label{midBallBound}
\sup_{B_{\WT T} \setminus \bigcup 3 \WT D_j} \abs{\tilde h}
&\ge \abs{\tilde h \circ \tilde g(z_0)}
= \abs{f(z_0)}
\ge \frac {e^{- L}}{1 + c_b \eps^2}.
\end{align}
An application of \eqref{propUB} shows that
\begin{align}
\label{bigBallBound}
\sup_{B_{\WT R} \setminus \bigcup 3 \WT D_j} \abs{\tilde h}
&\le \sup_{B_{\WT R}} \abs{\tilde h}
= \sup_{\tilde g\pr{B_R}} \abs{f \circ \tilde g^{-1}}
= \sup_{B_R} \abs{f}
\le \frac{1}{1 - c_b \eps^2} \sup_{B_R} \abs{v} 
\le \frac{e^{c_1 R}}{1 - c_b \eps^2} .
\end{align}
Combining \eqref{midBallBound} and \eqref{bigBallBound} then shows that
\begin{align*}
\sup_{B_{\WT T} \setminus \bigcup 3 \WT D_j} \abs{\tilde h}
&\ge \frac {1 - c_b \eps^2}{1 + c_b \eps^2} e^{- \pr{c_1 R + L}} \frac{e^{c_1 R}}{1 - c_b \eps^2}
\ge e^{- M}  \sup_{B_{\WT R} \setminus \bigcup 3 \WT D_j} \abs{\tilde h},
\end{align*}
where we introduce $M := c_1 R + L + \frac {c_d}{\log R}$ and $c_d$ depends on $c_b$ and $C_K$, see \eqref{epsDef}.
Because $R \ge R_2$ implies that $16 \WT R = 16 C_1 R \sqrt{\log R} < R^2$, then because $c_1 \ge 1$, $M > R \ge 2^{10}$, and we have $M > \log(16\WT R)$.
Therefore, all of the hypotheses of Proposition \ref{harmonicProp} hold hold with $h$, $\set{D_j}$, $R$, $S$, and $T$ replaced by $\tilde h$, $\set{\WT D_j}$, $\WT R$, $\WT S$, and $\WT T$, respectively. 

For $r \ll R$, an application of \eqref{MoriThm} shows that $g\pr{B_r}$ contains a disk of radius $r_1$, where
$$r_1 \ge \frac R {16} \pr{\frac r R}^K \ge  \frac R {16} \pr{\frac r R}^2.$$
It follows that $\tilde g\pr{B_r} \supset B_{\tilde r}$, where 
\begin{equation}
\label{scaleBounds}
\frac{16 \tilde r}{\WT R} = \pr{\frac r R}^2.
\end{equation}
If $r \le \sqrt{\frac{98 \cdot 16 R}{C_1 \sqrt{\log R}}}$, then $\tilde r \le 98$, so that $B_{\tilde r} \setminus \bigcup 3 \WT D_j = B_{\tilde r}$, and then
\begin{align}
\label{smallBallBound}
\sup_{B_{\tilde r} \setminus \bigcup 3 \WT D_j} \abs{\tilde h}
&= \sup_{B_{\tilde r}} \abs{\tilde h}
\le \sup_{\tilde g\pr{B_{r}}} \abs{f \circ \tilde{g}^{-1}}
= \sup_{B_{r}} \abs{f}.
\end{align}
Since $r \in \pr{0, \frac R {2^3}}$ implies that $\tilde r \in \pr{0, \frac{\WT R}{2^{10}}}$, then in this case we may apply Proposition \ref{harmonicProp} with $r$ replaced by $\tilde r$.
In particular, the conclusion \eqref{lowerBound} from Proposition \ref{harmonicProp} shows that
\begin{align}
\label{propApp}
\sup_{B_{\tilde r} \setminus \bigcup 3 \WT D_j} \abs{\tilde h}
\ge \pr{\frac{16 {\tilde r}} {\WT R}}^{\tilde \kappa} \sup_{B_{\WT T} \setminus \bigcup 3 \WT D_j} \abs{\tilde h},
\end{align}
where
\begin{align*}
\tilde \kappa
:= \kappa\pr{\WT R, \WT S, M}
= \max \set{6C_H C_1 R \sqrt{\log R}, 2^{13} \pr{c_1 R + L + \frac {c_d}{\log R}} RS^{-1}}.
\end{align*}
Since $R \ge 2^{10}$ implies $\frac R 8 > 1$, while $R \ge R_2$ implies $\sqrt{\frac{98 \cdot 16 R}{C_1 \sqrt{\log R}}} > 1$, then $r_0 := \min\set{\frac R 8, \sqrt{\frac{98 \cdot 16 R}{C_1 \sqrt{\log R}}}} > 1$.
Assume that $R \ge \hat R_0$ and $r < r_0$.
Combining \eqref{ufComparison}, \eqref{smallBallBound}, \eqref{propApp}, \eqref{scaleBounds}, and \eqref{midBallBound} shows that
\begin{align*}
\frac 1{1 - c_b \eps^2} \sup_{B_r} \abs{v}
&\ge \sup_{B_r} \abs{f}
\ge \sup_{B_{\tilde r} \setminus \bigcup 3 \WT D_j} \abs{\tilde h}
\ge \pr{\frac{16 {\tilde r}} {\WT R}}^{\tilde \kappa} \sup_{B_{\WT T} \setminus \bigcup 3 \WT D_j} \abs{\tilde h}
\ge \pr{\frac{r} {R}}^{2\tilde \kappa} \frac {e^{- L}} {1 + c_b \eps^2} .
\end{align*}
Since $C_1 =  \frac {a} {32} \sqrt{\frac{2 C_K }{\ln 2}}$, $\disp \frac{1 - c_b \eps^2}{1 + c_b \eps^2} \ge e^{- \frac{c_d}{\log R}}$, and $e < 4 \le \pr{2^{10}}^{\frac 1 5} \le \pr{\frac{R}{r}}^{\frac 1 5}$, then \eqref{propConc} follows, as required. 
\end{proof}

\section{Transformation maps}
\label{TransMaps}

In this section, we introduce the transformation maps that serve as real-valued versions of the conformal maps $z \mapsto z^\al$.
Once the maps are defined, we show how they transform solutions to elliptic PDEs and how they transform balls.
These results will be used at the beginning and at the end of the proof of Theorem \ref{mainThm}.

Given $z = (x,y) \in \R^2$ in Cartesian coordinates, the polar coordinates for $z$ are $(r, \te) \in \R_+ \times (-\pi, \pi]$, where $r = \sqrt{x^2 + y^2}$ and $\te = \sgn(y) \arccos\pr{\frac x r}$.
Let $\R^2_+ = \set{z = (r,\te) : r > 0, \te \in (- \frac \pi 2, \frac \pi 2)}$ denote the right half-plane.
For $\al \in (1, \iny)$, we define $T_\al : \R^2 \to \R^2$ to be the map associated to the conformal transformation on $\C$ given by $z \mapsto z^\al$. 
In polar coordinates, $T_\al$ is described as
\begin{equation}
\label{TalDefn}
T_\al(r, \te) = (r^\al, \al \te).
\end{equation}
To avoid continuity issues at $\te = \pi$, we restrict the domain and only consider $T_\al : \R^2_+ \to \R^2$.

\begin{lem}[Transformation of PDEs]
\label{transformEqLemma}
If $u : \R^2 \to \R$ is a solution to \eqref{ellipEq}, then with $v : \R^2_+ \to \R$ and $W : \R^2_+ \to \R$ defined by $v(z) = u(T_\al(z))$ and $W(z) = \al^2 \abs{z}^{2\al - 2} V(T_\al(z))$, it holds that
$$-\LP v + W v = 0 \; \text{ in } \R^2_+.$$
\end{lem}

\begin{proof}
Let $w := T_\al(z)$ be given in polar coordinates by $\pr{\rho, \vp}$.
Since $u$ is defined on $\R^2$, then in polar coordinates, $u(\rho, \vp)$ can be defined for all $\rho \in \R^+$ and all $\vp \in \R$ using periodicity.
We then see that $v(z) = v(r, \te) = u(r^\al, \al \te) = u(T_\al(z))$ is well-defined on $\R^2_+$. 
Since
\begin{align*}
\del_r v(r, \te) &= \al r^{\al - 1} \del_\rho u(\rho, \vp) \\
\del_{r}^2 v(r, \te) &= \al^2 r^{2\al - 2} \del_{\rho}^2 u(\rho, \vp) + \al \pr{\al - 1} r^{\al - 2} \del_\rho u(\rho, \vp) \\
\del_{\te}^2 v(r, \te) &= \al^2 \del_{\vp}^2 u(\rho, \vp),
\end{align*}
then
\begin{align*}
\LP v(z)
&= \del_{r}^2 v(r, \te) + \frac 1 r \del_{r} v(r, \te) + \frac 1 {r^2} \del_{\te}^2 v(r, \te) \\
&= \al^2 r^{2\al - 2} \del_{\rho}^2 u(\rho, \vp) + \al \pr{\al - 1} r^{\al - 2} \del_\rho u(\rho, \vp)
+ \al r^{\al - 2} \del_\rho u(\rho, \vp)
+ r^{-2} \al^2 \del_{\vp}^2 u(\rho, \vp) \\
&= \al^2 r^{2\al - 2} \brac{\del_\rho^2 u(\rho, \vp) + \rho^{-1} \del_\rho u(\rho, \vp) + \rho^{-2}\del_{\vp}^2 u(\rho, \vp)} \\
&= \al^2 r^{2\al - 2} \LP u(w) 
=  \al^2 r^{2\al - 2} V(w) u(w),
\end{align*}
where we have used \eqref{ellipEq}.
With $W(z) = \al^2 \abs{z}^{2\al - 2} V(T_\al(z)) = \al^2 r^{2\al - 2} V(w)$ as given, the conclusion follows.
\end{proof}

\begin{lem}[Ball containment]
\label{ballContainLemma}
Let $z_0 = (r_0, 0)$ and set $\disp \tilde r = \frac{r_0^{1 - \al}}{2 \sqrt 3 \al}$.
There exists $r_\al > 0$ so that whenever  with $r_0 \ge r_\al$, it holds that $T_\al(B_{\tilde r}(z_0)) \su B_1(T_\al(z_0))$ and $B_1(T_\al(z_0)) \su T_\al\pr{B_1(z_0)}$.
\end{lem}

\begin{proof}
With $\tilde r$ as given, it can be shown that
\begin{equation}
\label{ballIn} 
\begin{aligned}
B_{\tilde r}(z_0) 
&\su \set{r \in [r_0 - \tilde r, r_0 + \tilde r], \abs{\te} \le \frac {2\tilde r} {\sqrt 3 r_0}} \\
&= \set{r \in \brac{r_0\pr{1 -  \frac{r_0^{- \al}}{2 \sqrt 3 \al}}, r_0\pr{1 +  \frac{r_0^{- \al}}{2 \sqrt 3 \al}}}, \abs{\te} \le \frac {1} {3 \al r_0^\al}}.
\end{aligned}
\end{equation}
With $w_0 = T_\al(z_0)= \pr{r_0^\al, 0} = \pr{\rho_0, 0}$, it follows from \eqref{ballIn} that
\begin{align*}
T_\al(B_{\tilde r}(z_0)) 
&\su \set{r^\al \in \brac{r_0^\al\pr{1 -  \frac{r_0^{- \al}}{2 \sqrt 3 \al}}^\al, r_0^\al \pr{1 +  \frac{r_0^{- \al}}{2 \sqrt 3 \al}}^\al}, \abs{\al \te} \le \frac {\al} {3 \al r_0^\al}} \\
&= \set{\rho \in \brac{\rho_0 \pr{1 - \frac{\rho_0^{- 1}}{2 \sqrt 3 \al}}^\al, \rho_0 \pr{1 + \frac{\rho_0^{- 1}}{2 \sqrt 3 \al} }^\al}, \abs{\vp} \le \frac {1} {3 \rho_0} }.
\end{align*}
Since
\begin{align*}
\rho_0 \pr{1 \pm \frac{\rho_0^{- 1}}{2 \sqrt 3 \al} }^\al
&= \rho_0 \pm \frac{1}{2 \sqrt 3} + \frac{(1 - \frac 1 \al)}{24} r_0^{- \al} + \ldots 
\end{align*}
then there exists $r_1(\al) \gg 1$ so that whenever $r_0 \ge r_1$, we have $\rho_0 - \frac 1 2 \le \rho_0 \pr{1 - \frac{\rho_0^{- 1}}{2 \sqrt 3 \al} }^\al$ and $\rho_0 \pr{1 + \frac{\rho_0^{- 1}}{2 \sqrt 3 \al} }^\al \le \rho_0 +\frac 1 2$.
Because,
\begin{align}
&B_1(w_0)  \supset \set{\rho \in \brac{\rho_0 - \frac 1 2, \rho_0 + \frac 1 2}, \abs{\vp} \le \frac 1 {3\rho_0} },
\label{ballOf}
\end{align}
then it follows that $T_\al(B_{\tilde r}(z_0))  \su B_1(w_0)$.

On the other hand, with $T_\al^{-1}(B_1(w_0)) = \set{z : T_\al(z) \in B_1(w_0)}$, \eqref{ballIn} (with $\tilde r$ replaced by $1$) implies that 
\begin{align*}
T_\al^{-1}(B_1(w_0))
&\su \set{r^\al \in [\rho_0 - 1, \rho_0 + 1], \abs{\al \te} \le \frac {2} {\sqrt 3 \rho_0}} \\
&= \set{r \in \brac{\pr{r_0^\al - 1}^{\frac 1 \al}, \pr{r_0^\al + 1}^{\frac 1 \al}}, \abs{\te} \le \frac {2} {\sqrt 3 \al r_0^\al}}.
\end{align*}
Since
\begin{align*}
\pr{r_0^\al \pm 1}^{\frac 1 \al}
&= r_0 \pr{1 \pm \frac 1 {r_0^\al}}^{\frac 1 \al}
= r_0 \pm \frac 1 {\al r_0^{\al-1}} - \frac{1 - \frac 1 \al}{2 \al r_0^{2\al-1}}  + \ldots 
\end{align*}
then there exist $r_2(\al) \gg 1$ so that whenever $r_0 \ge r_2$, we have $\pr{r_0^\al + 1}^{\frac 1 \al} \le r_0 + \frac 1 2$ and $\pr{r_0^\al - 1}^{\frac 1 \al} \ge r_0 - \frac 1 2$.
If $r_0 \ge r_3 := \pr{\frac {2 \sqrt 3} { \al }}^{\frac 1 {\al - 1}}$, then
\begin{align*}
T_\al^{-1}(B_1(w_0))
&\su \set{r \in \brac{r_0 - \frac 1 2, r_0 + \frac 1 2}, \abs{\te} \le \frac {1} {3 r_0}}
\su B_1(z_0),
\end{align*}
where we have used \eqref{ballOf}.
It follows that $B_1(w_0) \su T^\al(B_1(z_0))$.
To complete the proof, choose $r_\al = \max\set{r_1, r_2, r_3}$.
\end{proof}

\section{Proof of Theorem \ref{mainThm}}
\label{MainProof}

In this section, we prove the main result, Theorem \ref{mainThm}.
Before proceeding to the rigorous details, we describe the main steps:
\begin{itemize}
\item Initialize the iterative argument by applying Theorem \ref{LandisGrowth} to the solution function $u$.
\item Compose the solution with the transformation map given in \eqref{TalDefn} to get a solution $v$ to a different Schr\"odinger equation, see Lemma \ref{transformEqLemma}.
\item Apply the iterative result, Proposition \ref{InductiveProp}, to $v$ many times to reduce the exponent.
\item Undo the change of variables to get the desired estimate for $u$.
\end{itemize}
We recall the following result originally proved in \cite{LMNN20}, see also \cite[Theorem 1.1]{Dav24} in the case where $N = 0$ for this formulation.
This estimate serves as the initialization step in our iteration argument.

\begin{thm}[Initialization step]
\label{LandisGrowth}
For some $a_0 > 0$, let $V : \R^2 \to \R$ satisfy $\norm{V}_{L^\iny} \le a_0^2$.
Assume that $u : \R^2 \to \R$ is a solution to \eqref{ellipEq} that satisfies \eqref{solNorm} and for each $w \in \R^2$,
$$\abs{u(w)} \le \exp\pr{c_0 \abs{w}}.$$
Then there exists constants $\overline{C}_0 = \overline{C}_0\pr{a_0, c_0} > 0$ and $\overline{R}_0 > 0$ so that whenever $R \ge \overline{R}_0$, it holds that
\begin{equation}
\label{uLower0}
\inf_{\abs{w_0} = R}\norm{u}_{L^\iny\pr{B_1(w_0)}} \ge \exp\pr{- \overline{C}_0 R \log^{\frac 3 2} R}.
\end{equation}
\end{thm}

Observe that if $u$ satisfies the hypothesis of Theorem \ref{mainThm}, then it also satisfies the hypotheses of Theorem \ref{LandisGrowth}.
We now have all of the tools we need to prove Theorem \ref{mainThm}.

\begin{proof}[Proof of Theorem \ref{mainThm}]

With $N \in (0, 2)$ as given, let $\al =\pr{1 - \frac N 2}^{-1} = \frac 2 {2-N} \in (1, \iny)$.
Set $\de = \min\set{\frac{\eps}{2 - N}, 1}$.
With $r_\al$ from Lemma \ref{ballContainLemma}, $\overline{C}_0(a_0, c_0)$ and $\overline{R}_0$ from Theorem \ref{LandisGrowth}, and $\hat R_0(a)$ from Proposition \ref{InductiveProp}, define $\bar{r}_1 \ge \max\set{r_\al, \overline{R}_0^{\frac 1 \al}, \hat R_0\pr{a_0 \al}}$ as small as possible so that each of the following conditions hold:
\begin{align}
& r^{\de} \ge \overline{C}_0 \al^{\frac 3 2} \log^{\frac 3 2} r
\label{rBig1} \\
&r \log r \ge c_d 
\label{rBig2} \\
& r^{\frac \de {1 + \de}} \ge \frac{3 a_0 \al C_H}{2^{14} \pr{5c_0 + 4} } \sqrt{\frac{C_K }{32\ln 2}} \sqrt{\log r}
\label{rBig3} \\
& r^{\frac {\de^2} 2 } \ge \brac{2^{14} \pr{5c_0 + 4} + \frac 2 5 }\log r
\label{rBig4} \\
& r^{\frac {5\de^2} 6 } \ge \brac{2^{14} \pr{5c_0 + 4} + \frac 2 5}\log \brac{\frac{4\al}{ \sqrt 3}  \pr{\frac 3 2 r}^{\al} }
\label{rBig5}.
\end{align}
Here $C_K$ and $c_d$ are universal constants from Proposition \ref{InductiveProp} and $C_H$ is the universal Harnack constant from Lemma \ref{discBounds}.

For each $k \in \N$, define 
\begin{equation*}
\al_k = \frac{2 + (k-1)N}{2 + (k-2) N} > 1.
\end{equation*}
Observe that $\al_1 = \al$, $\al_{k+1} = 2 - \frac 1 {\al_k} < \al_k$, and $\disp \lim_{k \to \iny} \al_k = 1$.
Choose $\ell \in \N$ so that $\frac{N}{2 + \pr{\ell - 2}N} \le \de < \frac{N}{2 + \pr{\ell - 3}N}$ or $\al_\ell \le 1 + \de < \al_{\ell - 1}$.
For each $k \in \set{1, 2, \ldots, \ell - 1}$, recursively define $\bar{r}_{k+1} = \bar{r}_k^{\al_k + \de} + \frac 1 2 \bar{r}_k - 1$, where $\bar r_1$ is from above.
Take $w_0 \in \R^2$ with $\abs{w_0} \ge R_0 := \bar{r}_\ell^\al$.
After a rotation, we may assume that $w_0$ belongs to the positive $x$-axis.
Choose $r_1 \ge \bar{r}_1$ so that with $r_{k+1} = r_k^{\al_k + \de} + \frac 1 2 r_k - 1$ for each $k \in \set{1, \ldots, \ell - 1}$, it holds that $\abs{w_0} = r_\ell^\al$.
Set $z_k = \pr{r_k, 0} \in  \R^2_+$ for each $k \in \set{1, \ldots, \ell - 1}$ so that
\begin{equation}
\label{zkNorm}
\abs{z_{k}} = \abs{z_{k-1}}^{\al_{k-1} + \de} + \frac 1 2 \abs{z_{k-1}} - 1.
\end{equation}

With $\al = \al_1 > 1$, let $T_\al : \R^2_+ \to \R^2$ be as given in \eqref{TalDefn}, see Section \ref{TransMaps}. 
Define $v : \R^2_+ \to \R$ by $v(z) = u(T_\al(z))$.

\begin{clm}
For every $k \in \set{1, \ldots, \ell -1}$, 
\begin{align}
\label{claim}
\norm{v}_{L^\iny\pr{B_1(z_k)}}
&\ge \exp\pr{- \abs{z_k}^{\al_k + \de}}.
\end{align}
\end{clm}

To prove the claim, we proceed by induction. 
Define $w_1 = T_\al(z_1) = (r_1^\al, 0)$. 
Since $r_1 \ge r_\al$, Lemma \ref{ballContainLemma} implies that $B_1(w_1) \su T_\al(B_1(z_1))$ and then
\begin{align*}
\norm{v}_{L^\iny\pr{B_1(z_1)}}
&= \norm{u \circ T_\al}_{L^\iny\pr{B_1(z_1)}}
= \norm{u }_{L^\iny\pr{T_\al(B_1(z_1))}}
\ge \norm{u}_{L^\iny\pr{B_1(w_1)}}.
\end{align*}
As $\abs{w_1} = r_1^\al \ge \overline{R}_0$, an application of \eqref{uLower0} from Theorem \ref{LandisGrowth} shows that
\begin{equation*}
\norm{u}_{L^\iny(B_1(w_1))} \ge \exp\pr{- \overline{C}_0 \abs{w_1} \log^{\frac 3 2} \abs{w_1}}.
\end{equation*}
Because $\abs{w_1} = \abs{z_1}^\al$, then $\overline{C}_0 \abs{w_1} \log^{\frac 3 2} \abs{w_1} = \overline{C}_0 \abs{z_1}^\al \log^{\frac 3 2} \abs{z_1}^\al= \overline{C}_0 \al^\frac 32 \abs{z_1}^\al \log^{\frac 3 2} \abs{z_1} \le \abs{z_1}^{\al +\de}$, where we have used \eqref{rBig1} to reach the last bound.
Combining these observations shows that
\begin{equation*}
\begin{aligned}
\norm{v}_{L^\iny\pr{B_1(z_1)}}
&\ge \norm{u}_{L^\iny\pr{B_1(w_1)}}
\ge \exp\pr{- \overline{C}_0 \abs{w_1} \log^{\frac 3 2} \abs{w_1}} 
\ge \exp\pr{- \abs{z_1}^{\al + \de}}.
\end{aligned}
\end{equation*}
In other words, \eqref{claim} holds for $k = 1$.

Now assume that \eqref{claim} holds for $k-1 \in \N$.
Define $S_k = \frac 1 2 \abs{z_{k-1}}$ and $R_k = \abs{z_{k-1}}^{\al_{k-1} +\de}$. 
Comparing with \eqref{zkNorm}, we see that $\abs{z_k} =  R_k + S_k - 1$.
Since $B_1(z_{k-1}) \su B_{R_k - S_k}(z_k)$, then the inductive hypothesis shows that there exists $z_0 \in \overline{B}_{R_k - S_k}(z_k)$ so that
\begin{align*}
\abs{v(z_0)}
\ge \exp\pr{- \abs{z_{k-1}}^{\al_{k-1} + \de}}
= \exp\pr{- R_k}.
\end{align*}
The bound in \eqref{uBound} shows that $\disp \abs{v(z)} = \abs{u(T_\al(z))} \le \exp\pr{c_0 \abs{z}^{\al\pr{1 - \frac N 2}}} = \exp\pr{c_0 \abs{z}}$ from which it follows that
\begin{align*}
\norm{v}_{L^\iny\pr{B_{R_k}(z_k)}}
\le \exp\pr{c_0 \pr{\abs{z_k} + R_k}}
\le \exp\pr{\frac {5c_0} 2 R_k}.
\end{align*}
Let $W : \R^2_+ \to \R$ be given by $W(z) = \al^2 \abs{z}^{2\al - 2} V(T_\al(z))= \al^2 \abs{w}^{2 - \frac 2 \al} V(w)$ so that by \eqref{Vbound},
\begin{align*}
\abs{W(z)} = \al^2 \abs{w}^{2 - \frac 2 \al} \abs{V(w)} \le a_0^2 \al^2 \abs{w}^{2 - N - \frac 2 \al} = \pr{a_0 \al}^2.
\end{align*}
An application of Lemma \ref{transformEqLemma} with assumption \eqref{ellipEq} shows that
\begin{equation*}
- \LP v + W v = 0 \; \text{ in } \R^2_+.
\end{equation*}
As $R_k \ge \bar{r}_1$ implies that $R_k \ge \hat R_0(a_0 \al)$ and $S_k \in \pr{2^8, \frac {R_k} 2}$, Proposition \ref{InductiveProp} is applicable with $a = a_0 \al$, $R= R_k$, $S= S_k$, $L = R_k$, and $c_1 =  \frac {5c_0} 2$.
With $r = 1$, Proposition \ref{InductiveProp} shows that
\begin{align}
\label{vkBd}
\norm{v}_{L^\iny(B_1(z_k))}
&\ge \exp\pr{-\tau_k(R_k) \log R_k},
\end{align}
where
\begin{equation}
\label{taukDefn}
\begin{aligned}
\tau_k(R_k) &= \max \set{3 a_0 \al C_H\sqrt{\frac{C_K }{32\ln 2}} R_k \sqrt{\log R_k}, 2^{15} \pr{\frac {5c_0} 2 + 1 + \frac {c_d}{R_k \log R_k}} R_k^{2 - \frac 1 {\al_{k-1} + \de}}} \\
&+\frac 1 5 \pr{R_k + \frac{c_d}{\log R_k}}.
\end{aligned}
\end{equation}
Because $R_k \ge \bar{r}_1$ and $\al_{k - 1} > 1$, conditions \eqref{rBig2} and \eqref{rBig3} imply that
\begin{align}
\label{tauKBound}
\tau_k(R_k) &\le \brac{2^{14} \pr{5c_0 + 4} + \frac 2 5} R_k^{2 - \frac 1 {\al_{k-1} + \de}} 
\end{align}
while condition \eqref{rBig4} implies that
$$ \brac{2^{14} \pr{5c_0 + 4} + \frac 2 5} \log R_k \le R_k^{\frac {\de^2} 2}.$$
Since $\al > 1$ and $\de \le 1$ implies that $2 - \frac{1}{\al + \de} + \frac{\de^2}2 < 2 - \frac 1 \al + \de$, then combining these bounds shows that
$$\tau_k(R_k) \log R_k \le R_k^{2 - \frac 1 {\al_{k-1}} + \de} = R_k^{\al_{k} + \de}.$$
Returning to \eqref{vkBd}, we conclude that
\begin{align*}
\norm{v}_{L^\iny(B_1(z_k))}
&\ge \exp\pr{-R_k^{\al_{k} + \de}}
\ge \exp\pr{-\abs{z_k}^{\al_{k} + \de}},
\end{align*}
establishing the claim given by \eqref{claim}.

Next we consider the behavior of $v$ on a ball centered at $z_\ell$.
The arguments in the previous paragraph show that Proposition \ref{InductiveProp} is applicable with $a = a_0 \al$, $R= R_\ell$, $S= S_\ell$, $L = R_\ell$, and $c_1 =  \frac {5c_0} 2$.
In particular, with $\disp r = \tilde r := \frac{\abs{z_\ell}^{1 - \al}}{2 \sqrt 3 \al}$, Proposition \ref{InductiveProp} shows that 
\begin{align*}
\norm{v}_{L^\iny(B_{\tilde r}(z_\ell))}
&\ge \pr{\frac{\tilde r}{R_\ell}}^{\tau_\ell(R_\ell)},
\end{align*}
where $\tau_\ell(R_\ell)$ is defined by \eqref{taukDefn} and the bound in \eqref{tauKBound} holds.
Because $R_\ell < \abs{z_\ell} < \frac 3 2 R_\ell$, then $\frac{R_\ell}{\tilde r} < \frac{4\al}{ \sqrt 3}  \pr{\frac 3 2 R_\ell}^{\al}$ and  condition \eqref{rBig5} implies that
$$\brac{2^{14} \pr{5c_0 + 4} + \frac 2 5} \log\pr{\frac{R_\ell}{\tilde r}} \le R_\ell^{\frac {5 \de^2}6 }.$$
Since $\al_{\ell - 1} > 1 + \de$ and $\de \le 1$ implies that $2 - \frac{1}{\al_{\ell - 1} + \de} + \frac {5 \de^2}6 < \al_\ell + \de \le 1 + 2 \de$, then $\tau_\ell(R_\ell) \log\pr{\frac{R_\ell}{\tilde r}} \le R_\ell^{1 + 2\de}$ and we conclude that
\begin{align*}
\norm{v}_{L^\iny(B_{\tilde r}(z_\ell))}
&\ge \exp\pr{- R_\ell^{1 + 2\de}}
\ge \exp\pr{-\abs{z_\ell}^{1 + 2\de}}.
\end{align*}
As $\abs{z_\ell} \ge r_\al$, an application of Lemma \ref{ballContainLemma} shows that $T_\al\pr{B_{\tilde r}(z_\ell)} \su B_1(T_\al(z_\ell))$ and then
\begin{align*}
\norm{u}_{L^\iny(B_{1}(T_\al(z_\ell)))}
\ge \norm{u}_{L^\iny(T_\al\pr{B_{\tilde r}(z_\ell)} )}
= \norm{v}_{L^\iny(B_{\tilde r}(z_\ell))}
&\ge \exp\pr{- \abs{z_\ell}^{1 + 2\de}}
= \exp\pr{-r_\ell^{1 + 2\de}}.
\end{align*}
Because we assumed that $w_0$ belongs to the positive $x$-axis and has $\abs{w_0} = r_\ell^\al$, then $T_\al(z_\ell) = \pr{r_\ell^\al, 0} = w_0$.
Since $\de = \frac{\eps}{2 - N}$, then $\frac{1 + 2\de}{\al} = \pr{1 - \frac N 2}\pr{1 + 2\de} = 1 - \frac N 2 + \eps$ and $r_\ell^{1 + 2\de} = \abs{w_0}^{1 - \frac N 2 + \eps}$.
Combining these observations shows that $\norm{u}_{L^\iny(B_{1}(w_0))} \ge \exp\pr{-\abs{w_0}^{1 - \frac N 2 + \eps}}$.
Since $w_0 \in \R^2$ with $\abs{w_0} \ge R_0$ was arbitrary, then \eqref{uLower} follows, as required.
\end{proof}

\bibliography{refs}
\bibliographystyle{alpha}

\end{document}